\documentclass[preprint,12pt]{elsarticle}
\usepackage{amsthm,amsfonts,amssymb,amscd,amsmath,enumerate,verbatim,calc,graphicx,geometry}
\usepackage[all]{xy}

\newtheorem{theorem}{Theorem}[section]
\newtheorem{lemma}[theorem]{Lemma}

\newtheorem{corollary}[theorem]{Corollary}
\theoremstyle{definition}
\theoremstyle{definitions}
\newtheorem{definition}[theorem]{Definition}

\theoremstyle{notations}

\theoremstyle{remarks}

\journal{ }

\begin{document}

\begin{frontmatter}
\title{The Capacity of Some Classes of Polyhedra}
\author[]{Mojtaba~Mohareri}
\ead{m.mohareri@stu.um.ac.ir}
\author[]{Behrooz~Mashayekhy\corref{cor1}}
\ead{bmashf@um.ac.ir}
\author[]{Hanieh~Mirebrahimi}
\ead{h$_{-}$mirebrahimi@um.ac.ir}
\address{Department of Pure Mathematics, Center of Excellence in Analysis on Algebraic Structures, Ferdowsi University of
Mashhad,\\
P.O.Box 1159-91775, Mashhad, Iran.}
\cortext[cor1]{Corresponding author}
\begin{abstract}
 K. Borsuk in 1979, in the Topological Conference in Moscow, introduced the concept of the capacity of a compactum.  In this paper, we compute the capacity of the product of two spheres of the same or different dimensions and the capacity of lense spaces. Also, we present an upper bound for the capacity of a  $\mathbb{Z}_n$-complex, i.e.,  a connected finite 2-dimensional CW-complex with finite cyclic fundamental group $\mathbb{Z}_n$.
\end{abstract}


\end{frontmatter}
\section{Introduction and Motivation}
In this paper, every CW-complex is assumed to be finite and connected. Also, every map between two CW-complexes is assumed to be cellular. We assume that the reader is familiar with the basic notions and facts of homotopy theory.

 K. Borsuk in \cite{So}, introduced the concept of the capacity of a compactum as follows:
 the capacity $C(X)$ of a compactum $X$ is the cardinality of the set of all shapes of compacta $A$  which are shape dominated by $X$ (for more details, see \cite{Mar}).

For polyhedra,  the notions shape and shape domination in the above definiton can be replaced by the notions homotopy type and homotopy domination, respectively. Indeed, by some known
results in shape theory one can conclude that for any polyhedron $P$, there is a one to one functorial correspondence between the shapes of compacta shape dominated by $P$ and the homotopy types of CW-complexes (not necessarily finite) homotopy dominated by $P$ (see \cite{1}).

S. Mather in \cite{Mather} proved that every polyhedron dominates only countably many  different homotopy types (hence shapes). Note that the capacity of a compactum is a homotopy invariant, i.e., for
compacta $X$ and $Y$ with the same homotopy type, $C(X)=C(Y)$. This property can be useful for  distinguishing two compacta up to homotopy equivalence. Hence it seems interesting to find compacta with finite capacity and compute the capacity of some of their well-known spaces.  Borsuk in \cite{So} asked a question: `` Is it true that the capacity of every finite polyhedron is finite? ''.  D. Kolodziejczyk in \cite{3} gave a negative answer to this question. Also, she investigated some conditions for polyhedra to have finite capacity (\cite{1, 2}). For instance, a polyhedron $Q$ with finite fundamental group $\pi_1 (Q)$ and a polyhedron $P$ with abelian fundamental group $\pi_1 (P)$ and finitely generated homology groups $H_i (\tilde{P})$, for $i\geq 2$ where $\tilde{P}$ is the universal cover of $P$, have finite capacities.

Borsuk in \cite{So} mentioned that the capacities of $\bigvee_k \mathbb{S}^1$ and $\mathbb{S}^n$ are equal to $k+1$ and 2, respectively.  The authors in \cite{Moh} computed the capacities of Moore spaces $M(A,n)$ and Eilenberg-MacLane spaces $K(G,n)$. In fact, we showed that the capacities of a Moore space $M(A,n)$ and an Eilenberg-MacLane space $K(G,n)$ are equal to the number of direct summands of $A$ and semidirect factors of  $G$ up to isomorphism, respectively. Also, we computed the capacity of the  wedge sum of finitely many Moore spaces of different degrees and the capacity of the product of finitely many Eilenberg-MacLane spaces of different homotopy types. In \cite{Moh1}, we showed that the capacity of $\bigvee_{n\in I} (\vee_{i_n} \mathbb{S}^n)$ is equal to $\prod_{n\in I}(i_n +1)$ where $\vee_{i_n} \mathbb{S}^n$ denotes the wedge sum of $i_n$ copies of $\mathbb{S}^n$, $I$ is a finite subset of $\mathbb{N}$ and $i_n \in \mathbb{N}$. In fact, we proved  that every space homotopy dominated by $\bigvee_{n\in I} (\vee_{i_n} \mathbb{S}^n)$ has the homotopy type of $\bigvee_{n\in I} (\vee_{j_n} \mathbb{S}^n)$, where  $0\leq j_n \leq i_n$.

M. Abbasi et al. in \cite{Mah} computed the capacity of 2-dimensional manifolds. They showed that the capacities of a compact orientable surface of genus $g\geq 0$ and a compact non-orientable surface of genus $g>0$ are equal to $g+2$ and $[\frac{g}{2}]+2$, respectively. In \cite{Moh1}, we proved the capacity of a 2-dimensional CW-complex $P$ with free fundamental group $\pi_1 (P)$ is finite and is equal to $(rank\; \pi_1 (P)+1)\times (rank\; H_2 (P)+1)$.

 We outline the main results of the paper. In Section 3,  we compute the capacity of the product of two spheres of the same or different dimensions. Then, in Section 4, we compute the capacity of lens spaces which are a class of closed orientable 3-manifolds. Also, by a similar method to computation of capacity of a  lens space, we show that the capacity of a real projective $n$-space is equal to 2. Note that this result was proved by Y. Kodama et al. in \cite{Kod} in a different manner. Finally, in Section 5, we find an upper bound for the capacity of a $\mathbb{Z}_n$-complex, a 2-dimensional CW-complex with finite cyclic fundamental group $\mathbb{Z}_n$. In fact, we show that every space homotopy dominated by a $\mathbb{Z}_n$-complex $P$ where $n=p_{1}^{\alpha_1}p_{2}^{\alpha_2}\cdots p_{m}^{\alpha_m}$ (for mutually distinct primes $p_i$ and positive integers  $\alpha_i$)   has the homotopy type of a $\mathbb{Z}_{m}$-complex where $m=p_{i_1}^{\alpha_{i_1}}\cdots p_{i_j}^{\alpha_{i_j}}$ for $i_1 ,\cdots ,i_j \in \{ 1,\cdots ,m\}$.

\section{Preliminaries}
We recall here only some facts that we will use  throughout the paper.
\begin{definition}\cite{Wall}.
Let $P$ be a CW-complex. The condition $\mathcal{D}_n$ on $P$ are defined as follows:

$\mathcal{D}_n$: $H_i(\tilde{P})=0$ for $i>n$, and $H^{n+1}(P;\mathcal{B})=0$ for all coefficient bundles $\mathcal{B}$ (for more detalis, see \cite{STE}). Note that $\tilde{P}$ denotes the universal covering space of $P$.
\end{definition}
 \begin{definition}\citep{Hat}.
A Moore space of degree $n$ $(n\geq 2)$ is a simply connected $CW$-complex $X$ with a single non-vanishing homology group of degree $n$, that is $\tilde{H}_{i}(X,\mathbb{Z})=0$ for $i\neq n$. A Moore space of degree $n$ is denoted by $M(A,n)$ where $A\cong \tilde{H}_{n} (X,\mathbb{Z})$.
\end{definition}
As an example, the $n$-sphere $\mathbb{S}^n$ for  $n\geq 2$ is a Moore space of degree $n$, $\mathbb{S}^n = M(\mathbb{Z}, n)$.
\begin{theorem}\label{-2}\cite{Hat}.
The homotopy type of a  Moore space $M(A,n)$ is uniquely determined by $A$ and $n$ for $n \geq 2$.
\end{theorem}
\begin{theorem}\label{1}\cite[Theorem 4.32]{Hat}
(Hurewicz theorem) If a topological space $X$ is $(n-1)$-connected, $n\geq 2$, then $\tilde{H}_i (X)=0$ for $i<n$, $h_{n}^{X}:\pi_n (X)\cong H_n (X)$ and $h_{n+1}^{X}:\pi_{n+1}(X)\longrightarrow H_{n+1}(X)$ is epimorphism, where $h_{i}^{X}:\pi_i (X)\longrightarrow H_i (X)$ denotes the $i$-th Hurewicz map.
\end{theorem}
\begin{theorem}\cite[page 91]{Hil} .\label{44}
For all $2\leq r\leq p+q+\min \{ p,q\} -3$, we have
\[
\pi_r (\mathbb{S}^p \vee \mathbb{S}^q )\cong \pi_r (\mathbb{S}^p )\oplus \pi_r (\mathbb{S}^q )\oplus \pi_r (\mathbb{S}^{p+q-1}).
\]
\end{theorem}
\begin{theorem}\cite[Proposition 2.6.15]{Bau1}.\label{2}
A simply connected space $X$ is homotopy equivalent to a one point union of Moore spaces if and only if $h_{n}^{X}$ is split surjective for all $n>1$.
\end{theorem}
Recall that an epimomorphism $f:G\longrightarrow H$ is called split surjective if the following short exact sequence is split
\[
0\rightarrow Ker\: f\hookrightarrow G\xrightarrow{f} H \rightarrow 0.
\]
\begin{theorem}\cite{Ber}.\label{4}
Let $X$ be a topological space which is homotopy dominated by a closed (compact without boundary) connected topological $n$-dimensional manifold $M$. If $H^n (X;\mathbb{Z}_2 )\neq 0$, then $X$ has the homotopy type of $M$.
\end{theorem}
\begin{lemma}\cite[Example 2.43]{Hat}.\label{K}
Let $X$ be the Eilenberg-MacLane space $K(\mathbb{Z}_m ,1)$. Then $H_n (X)\cong\mathbb{Z}_m$ for odd $n$ and $H_n (X)\cong 0$ for even $n>0$.
\end{lemma}
\section{The Capacity of Product of Two Spheres}
In this section, we compute the capacity of product of two spheres of the same or different dimensions.
\begin{lemma}\label{100}
The capacity of $\mathbb{S}^1 \times \mathbb{S}^n$ is equal to $4$, for $n\geq 2$.
\end{lemma}
\begin{proof}
Put $P=\mathbb{S}^1 \times \mathbb{S}^n$. By the Kunneth formula (see, for example, \cite[Theorem 3B.6]{Hat}), we know that
\[
H_i (P)\cong
\begin{cases}
\mathbb{Z}, & i=0,1,n,n+1 \\
0, & \text{otherwise}.
\end{cases}
\]
Suppose that $X$ is homotopy dominated by $P$ and $\tilde{X}$  denotes the universal covering space of $X$. Then $\pi_1 (X)$ and $H_{i}(X)$ are isomorphic to a direct suumand of $\pi_1 (P)\cong \mathbb{Z}$ and of $H_i (P)$, respectively, for all $i\geq 2$.

First, let $H_{n+1}(X)\cong \mathbb{Z}$. By the universal coefficient theorem for cohomology (see, for example, \cite[Theorem 3.2]{Hat}), we have
\[
H^{n+1} (X;\mathbb{Z}_2 )\cong Hom (H_{n+1}(X),\mathbb{Z}_2 )\cong \mathbb{Z}_2 \neq 0.
\]
Then, since $P$ is a closed compact connected $n+1$-dimensional topological manifold, $X$ and $P$ have the same homotopy type  by Theorem \ref{4}.

Second, let $H_{n+1} (X)=0$. Then we just have  the following cases:

Case One. $\pi_1 (X)=1$ and $H_i (X)=0$ for all $i\geq 2$. Then  by the Whitehead Theorem (see \cite[Corollary 4.33]{Hat}), $X$ and  $\{ *\}$ have the same homotopy type.

Case Two. $\pi_1 (X)=1$, $H_n (X)\cong \mathbb{Z}$ and $H_i (X)=0$ for all $i\neq n$. Then $X$ is the Moore space $M(\mathbb{Z},n)$ and so, $X$ has the homotopy type of  $\mathbb{S}^n$.

Case Three. $\pi_1 (X)\cong \mathbb{Z}$. We know that $\tilde{P}=\mathbb{R}\times \mathbb{S}^n$ is the universal covering space of $P=\mathbb{S}^1 \times \mathbb{S}^n$ and $\tilde{X}$ is homotopy dominated by $\tilde{P}$. Since the capacity of a compactum is a homotopy invariant, so  $C(\tilde{P})=C(\mathbb{S}^n )=2$ and so  $\tilde{X}$ has the homotopy type of  $\{ *\}$ or $\tilde{P}$. If $\tilde{X}$ and $\tilde{P}$ have the same homotopy type, then the domination map $d_X :P\longrightarrow X$  induces isomorphims
\[
d_{X*} :\pi_1 (P)\longrightarrow \pi_1 (X),\quad \tilde{d}_{X*} : H_i (\tilde{P})\longrightarrow H_i (\tilde{X})
\]
 for all $i\geq 2$ (note that an epimorphism between isomorphic Hopfian groups is an isomorphism). Then by the Whitehead Theorem (see, for example, \cite[Theorem 3.7, p. 113]{Hil}), $d_X$ is a homotopy equivalence. Thus $X$ and $P$ have the same homotopy type which is a contradiction because $H_{n+1}(X)=0$. Hence $\tilde{X}$ has the homotopy type of  $\{ *\}$ and so, $X$ is the Eilenberg-MacLane space $K(\mathbb{Z},1)$. This shows that  $X$ has the homotopy type of $\mathbb{S}^1$.
\end{proof}
\begin{lemma}\label{101}
The capacity of $\mathbb{S}^n \times \mathbb{S}^n$ is equal to $3$, for all $n\geq 1$.
\end{lemma}
\begin{proof}
For the case $n=1$, $\mathbb{S}^1 \times \mathbb{S}^1$ is a product of Eilenberg-MacLane spaces and by  \cite[Proposition 4.6]{Moh},  $C(\mathbb{S}^1 \times \mathbb{S}^1 )=3$.

Let $n\geq 2$ and $P=\mathbb{S}^n \times \mathbb{S}^n$. By the Kunneth formula, we know that
\[
H_i (P)\cong
\begin{cases}
\mathbb{Z}, & i=0,2n\\
\mathbb{Z}\times \mathbb{Z}, & i=n\\
0 & \text{otherwise}.
\end{cases}
\]
Suppose $X$ is homotopy dominated by $P$. Then  $H_{i}(X)$ is isomorphic to a direct suumand of $H_i (P)$  for all $i\geq 2$.

First, let $H_{2n} (X)\cong \mathbb{Z}$. By the universal coefficient theorem for cohomology, we have $H^{2n} (X;\mathbb{Z}_2 ) \neq 0$. Then, since $P$ is a closed compact connected $2n$-dimensional topological manifold, $X$ and $P$ have the same homotopy type  by Theorem \ref{4}.

Second, let $H_{2n} (X)=0$. Then we just have  the following cases:

Case One. $H_i (X)=0$ for all $i\geq 1$. Then by the Whitehead Theorem, $X$ and  $\{ *\}$ have the same homotopy type.

Case Two. $H_n (X)\cong \mathbb{Z}$ and $H_i (X)=0$ for all $i\neq n$. Then $X$ is the Moore space $M(\mathbb{Z},n)$ and so, $X$ and  $\mathbb{S}^n$ have the same homotopy type.

Case Three. $H_n (X)\cong \mathbb{Z}\times \mathbb{Z}$ and $H_i (X)=0$ for all $i\neq n$. Then $X$ has the homotopy type of $\mathbb{S}^n \vee \mathbb{S}^n$. But $\mathbb{S}^n \vee \mathbb{S}^n$ is not homotopy dominated by $P=\mathbb{S}^n \times \mathbb{S}^n$ because $\pi_{2n-1} ( \mathbb{S}^n \vee \mathbb{S}^n )\cong \pi_{2n-1}(\mathbb{S}^n )\times \pi_{2n-1}(\mathbb{S}^n )\times \mathbb{Z}$ (by Theorem \ref{44}) is not isomorphic to a direct summand of $\pi_{2n-1}(\mathbb{S}^n \times \mathbb{S}^n)\cong \pi_{2n-1}(\mathbb{S}^n )\times \pi_{2n-1}(\mathbb{S}^n )$. Thus this case does not occur.
\end{proof}
\begin{lemma}\label{3}
Let $X$ be a simply connected CW-complex with $H_n (X)\cong H_{n+1}(X)\cong \mathbb{Z}$ and $H_i (X)=0$ for all $i\neq n,n+1$. Then $X$ has the homotopy type of $\mathbb{S}^n \vee \mathbb{S}^{n+1}$.
\end{lemma}
\begin{proof}
We consider three following cases:
\begin{enumerate}[(i).]
\item
Let $n=2$. Since $X$ is  simply connected (equivalently, 1-connected), then  $h_{2}^{X}:\pi_2 (X)\cong H_2 (X)$ and $h_{3}^{X}:\pi_3 (X)\longrightarrow H_3 (X)$ is an epimorphism by Theorem \ref{1}. Clearly, $h_{2}^{X}$ is a split surjective homomorphism. Also, since
\[
0\rightarrow Ker\: h_3 \hookrightarrow \pi_3 (X)\xrightarrow{h_{3}^{X}} H_3 (X) \rightarrow 0
\]
is a short exact sequence and $H_3 (X)\cong \mathbb{Z}$ is a projective $\mathbb{Z}$-module,  $h_{3}^{X}$ is also split surjective. On the other hand, since $H_i (X)=0$ for all $i\geq 4$, $h_{i}^{X}$ is also split surjective. Thus by Theorem \ref{2}, the space $X$ is homotopy equivalent to a wedge of Moore spaces. By hypothesis $H_2 (X)\cong H_3 (X)\cong \mathbb{Z}$, so  $X$ has the homotopy type of $\mathbb{S}^2 \vee \mathbb{S}^3$.
\item
Let $n=3$. Since $X$ is simply connected (equivalently, 1-connected), $\pi_2 (X)\cong H_2 (X)$ by Theorem \ref{1}.  By hypothesis, we have $H_2 (X)=0$. Hence $\pi_2 (X)=0$ and so, $X$ is a 2-connected space. Again by Theorem \ref{1}, $h_{3}^{X}:\pi_3 (X)\cong H_3 (X)$ and $h_{4}^{X}:\pi_4 (X)\longrightarrow H_4 (X)$ is an epimorphism. Now, similar to the previous argument, one can prove that $X$ has the homotopy type of $\mathbb{S}^3 \vee \mathbb{S}^{4}$.
\item
Let $n>3$. Similar to the argument mentioned in the previous case, one can easily see that $X$ is $(n-1)$-connected. By Theorem \ref{1},  $h_{n}^{X}:\pi_n (X)\cong H_n (X)$ and $h_{n+1}^{X}:\pi_{n+1}(X)\longrightarrow H_{n+1}(X)$ is an epimorphism. Now, similar to the argument mentioned in case (i), one can prove that $X$ has the homotopy type of $\mathbb{S}^n \vee \mathbb{S}^{n+1}$.
\end{enumerate}
\end{proof}
\begin{lemma}\label{102}
The capacity of  $\mathbb{S}^n \times \mathbb{S}^{n+1}$ is equal to $4$, for all $n\geq 2$.
\end{lemma}
\begin{proof}
Put $P=\mathbb{S}^n \times \mathbb{S}^{n+1}$. By the Kunneth formula, we have
\[
H_i (P)\cong
\begin{cases}
\mathbb{Z}, &   i=0,n,n+1,2n+1,\\
0, & \text{otherwise}.
\end{cases}
\]
Suppose that $X$ is homotopy dominated by $P$. Then  $H_{i}(X)$ is isomorphic to a direct suumand of $H_i (P)$  for all $i\geq 2$.

First, let $H_{2n+1}(X)\cong \mathbb{Z}$. By the universal coefficient theorem for cohomology, we have $H^{2n+1} (X;\mathbb{Z}_2 ) \neq 0$. Then, since $P$ is a closed compact connected $(2n+1)$-dimensional topological manifold, $X$ and $P$ have the same homotopy type by Theorem \ref{4}.

Second, let $H_{2n+1}(X)=0$. We just have  the following cases:

Case One. $H_i (X)=0$ for all $i$. Then by the Whitehead Theorem, $X$ and $\{ *\}$ have the same homotopy type.

Case Two. $H_n (X)\cong \mathbb{Z}$ and $H_i (X)=0$ for all $i\neq n$. Then $X$ is the Moore space $M(\mathbb{Z},n)$ and so,  $X$ and  $\mathbb{S}^n$ have the same homotopy type.

Case Three. $H_{n+1} (X)\cong \mathbb{Z}$ and $H_i (X)=0$ for all $i\neq 3$. Then $X$ is the Moore space $M(\mathbb{Z},n+1)$ and so,  $X$ and  $\mathbb{S}^{n+1}$ have the same homotopy type.

Case Four.  $H_n (X)\cong H_{n+1} (X)\cong \mathbb{Z}$ and $H_i (X)=0$ for all $i\neq 2,3$. Then by Lemma \ref{3}, $X$ and  $\mathbb{S}^n \vee \mathbb{S}^{n+1}$ have the same homotopy type. But $\mathbb{S}^n \vee \mathbb{S}^{n+1}$ is not homotopy dominated by $\mathbb{S}^n \times \mathbb{S}^{n+1}$ because $\pi_{2n} (\mathbb{S}^n \vee \mathbb{S}^{n+1} )$ is not isomorphic to a subgroup of $\pi_{2n} (\mathbb{S}^{n} \times \mathbb{S}^{n+1} )$. Thus this case does not happen.
\end{proof}
\begin{lemma}\label{103}
The capacity of $\mathbb{S}^n \times \mathbb{S}^m$ is equal to $4$, where $2<n+1<m$.
\end{lemma}
\begin{proof}
Put $P=\mathbb{S}^n \times \mathbb{S}^{m}$. By the Kunneth formula, we know that
\[
H_i (P)\cong
\begin{cases}
\mathbb{Z}, &  i=n,m,n+m,\\
0, & \text{otherwise}.
\end{cases}
\]
 Then  $H_{i}(X)$ is isomorphic to a direct suumand of $H_i (P)$  for all $i\geq 2$.

 First, let $H_{n+m}(X)\cong \mathbb{Z}$. By the universal coefficient theorem for cohomology, $H^{n+m} (X;\mathbb{Z}_2 ) \neq 0$. Then, since $P$ is a closed compact connected $(n+m)$-dimensional topological manifold, $X$ and $P$ have the same homotopy type by Theorem \ref{4}.

Second, let $H_{n+m}(X)=0$. We have  the following cases:

Case One. $H_i (X)=0$ for all $i$; Then by the Whitehead Theorem, $X$ and  $\{ *\}$ have the same homotopy type.

Case Two.  $H_n (X)\cong \mathbb{Z}$ and $H_i (X)=0$ for all $i\neq n$; Then $X$ is the Moore space $M(\mathbb{Z},n)$, and so  $X$ and $\mathbb{S}^n$ have the same homotopy type.

Case Three. $H_m (X)\cong \mathbb{Z}$ and $H_i (X)=0$ for all $i\neq m$; Then $X$ is the Moore space $M(\mathbb{Z},m)$, and so $X$ and $\mathbb{S}^m$ have the same homotopy type.

Case Four. $H_n (X)\cong H_m (X)\cong \mathbb{Z}$ and $H_i (X)=0$ for all $i\neq n,m$. We know that $\mathbb{S}^n$ and $\mathbb{S}^m$ have  CW decompositions $\{ a ,e^n \}$ and $\{ b,  e^m \}$ respectively, where $a$ and $b$ are 0-cells. Hence $P=\mathbb{S}^n \times \mathbb{S}^m$ has a CW decomposition $\{ a\times b, a\times e^m , e^n \times b, e^n \times e^m \}$. One can consider $\mathbb{S}^n \vee \mathbb{S}^m$ as the subspace $\mathbb{S}^n \times \{ b\} \cup \{ a\} \times \mathbb{S}^m$ of $P=\mathbb{S}^n \times \mathbb{S}^m$. Then $\mathbb{S}^n \vee \mathbb{S}^m$ has a CW decomposition $\{ a\times b, a\times e^m , e^n \times b\}$. Since $2<n+1<m$, then $H_n (P)=H_n (X)=C_n (X)$ and $H_m (P)=H_m (X)=C_m (X)$. Hence the homomorphism $i_* :H_i (\mathbb{S}^n \vee \mathbb{S}^m )\longrightarrow H_i (P)$ is identity for $i=n,m$, where $i:\mathbb{S}^n \vee \mathbb{S}^m \hookrightarrow P$ is the inclusion map. Now, consider the map
\[
h=i\circ d_X :\mathbb{S}^n \vee \mathbb{S}^m \longrightarrow X,
\]
 where $d_X :P\longrightarrow X$ is the domination map. It is easy to see that $(d_X)_* :H_i (P)\longrightarrow H_n (X)$ is an isomorphism for $i=n,m$ (epimorphism between two isomorphic Hopfian groups). Hence, the map $h_* :H_i (\mathbb{S}^n \vee \mathbb{S}^m )\longrightarrow H_i (X)$ is an isomorphism for all $i\geq 2$ and so by the Whitehead Theorem, $h:\mathbb{S}^n \vee \mathbb{S}^m \longrightarrow X$ is a homotopy equivalence. This shows that $\mathbb{S}^n \vee \mathbb{S}^m$ is  homotopy dominated by $P$ which is a contradiction because by Theorem \ref{44}, $\pi_{n+m-1} (\mathbb{S}^n \vee \mathbb{S}^m )$ is not isomorphic to a subgroup of $\pi_{n+m-1} (\mathbb{S}^{n} \times \mathbb{S}^m )$. Thus this case does not happen.
\end{proof}
Finally by Lemmas \ref{100}, \ref{101}, \ref{102} and \ref{103}, we can conclude the following theorem which is the main result of this section.
\begin{theorem}\label{P}
For $n,m\geq 1$, the capacity of $\mathbb{S}^n \times \mathbb{S}^m$ is equal to $4$ if $n\neq m$ and it is equal to $3$ if $n=m$.
\end{theorem}
\section{The Capacity of Lens spaces}
In this section, we use reference \cite{Hat1} for expressing  the classification of closed (compact and without bounday) orientable 3-manifolds.

By Kneser's Theorem, every compact orientable 3-manifold $M$ factors as a connected sum of primes, $M = P_1 \# \cdots \# P_n$, and this decomposition is unique up to insertion or deletion of $\mathbb{S}^3$ summands. On the other hand, primes 3-manifold that are closed and orientable can be lumped broadly into three classes:
\begin{enumerate}[\textbf{Type} I:]
\item
\textbf{infinite cyclic fundamental group.} Only one such manifold (which is closed and orientable) is $\mathbb{S}^1 \times \mathbb{S}^2$. The capacity of such manifold is 4 by Theorem \ref{P}.
\item
\textbf{infinite noncyclic fundamental group.}  Such a manifold M is an Eilenberg-MacLane space $K(\pi , 1)$. It has been shown that $\pi$ is a finitely generated torsion-free group. By \cite[Proposition 4.4]{Moh}, the capacity of an Eilenberg-MacLane space $K(\pi ,1)$ is equal to the number of r-images of $\pi$ up to isomorphism. Recall that a group homomorphism $f :G \longrightarrow H$ is an r-homomorphism if there
exists a homomorphism $g :H\longrightarrow G$ such that $f\circ g = id_H$. Then $H$ is an r-image of $G$. In particular, if $\pi$ is a nilpotent group, then $\pi$ has only finitely many r-images up to isomorphism by \cite[Corollary 1]{1}.
\item
\textbf{finite fundamental group.} For such a manifold $M$, the universal covering space $\tilde{M}$ is
simply connected and closed, hence a homotopy sphere. All these manifolds (which
are spherical 3-manifolds) have the form $M =\mathbb{S}^3 /\Gamma$ for $\Gamma$ a finite subgroup of $SO(4)$ acting freely on $\mathbb{S}^3$ by rotations. Thus $\mathbb{S}^3$ is the universal covering space of $M$ and $\Gamma =\pi_1 (M)$. The spherical manifolds with cyclic fundamental group $\pi_1 (M)=\Gamma$ are the lens spaces which are defined as follows:

Let $p$ and $q$ be relatively prime integers. Regard $\mathbb{S}^3$ as all $(z_0 ,z_1 )\in \mathbb{C}^2$ with $|z_0 |^2 + |z_1 |^2 = 1$. Let $\zeta =e^{2\pi i/p}$ be a primitive $p$th root of
unity; define $h: \mathbb{S}^3 \longrightarrow \mathbb{S}^3$ by $h(z_0 ,z_1 )=(\zeta z_0 ,\zeta^q z_1 )$,  and define an equivalence relation on $\mathbb{S}^3$ by $(z_0 ,z_1 )\sim (z'_0 ,z'_1 )$ if there exists an  integer $m$ with $h^m (z_0, z_1 ) = (z'_0 ,z'_1 )$. The quotient space $\mathbb{S}^3 /\sim$ is called a lens space and is denoted by $L(p,q)$. Note that $\pi_1 (L(p,q))\cong \mathbb{Z}_p$. In the next theorem, we compute the capacity of lens spaces.
\end{enumerate}
\begin{theorem}
The capacity of the lens space $L(p,q)$ is equal to $2$, where $p$ and $q$ are relatively prime integers.
\end{theorem}
\begin{proof}
Suppose $X$ is homoyopy dominated by $P=L(p,q)$ with the domination map $d_X :P\longrightarrow X$ and $\tilde{X}$  denotes the universal covering space of $X$. Then $\pi_1 (X)$ and $H_{i}(X)$ are isomorphic to a direct suumand of $\pi_1 (P)\cong \mathbb{Z}$ and of $H_i (P)$, respectively, for all $i\geq 2$. We know that
\[
H_i (P)\cong
\begin{cases}
\mathbb{Z}, & i=0,3 \\
\pi_1 (P)\cong \mathbb{Z}_p , & i=1\\
0, & \text{otherwise}.
\end{cases}
\]
 We have the following cases:

Case One. $\pi_1 (X)=1$ and $H_i (X)=0$ for all $i\geq 2$.  Then by the Whitehead Theorem $X$ and $\{ *\}$ have the same homotopy type.

Case Two. $\pi_1 (X)=1$ and  $H_3 (X)\cong \mathbb{Z}$ and $H_i (X)=0$ for all $i\neq 3$. Then $H^3 (X;\mathbb{Z}_2)\neq 0$ and so, $X$ and $L(p,q)$ have the same homotopy type (by Theorem \ref{4}) which is a contradiction because $\pi_1 (X)=1$. Thus this case does not occur.

Case Three. $\pi_1 (X)\cong \mathbb{Z}_p$. Then $(d_X )_* :\pi_1 (P)\longrightarrow \pi_1 (X)$ is an  isomorphism.  Since $\tilde{X}$ is homotopy dominated by $\tilde{P}=\mathbb{S}^n$ and $C(\mathbb{S}^3 )=2$, then  $\tilde{X}$ has the homotopy type of $\{ *\}$ or $\mathbb{S}^3$. If $\tilde{X}$ has the homotopy type of $\{ *\}$, then $X$ is the Eilenberg-MacLane space $K(\mathbb{Z}_p ,1)$. By Lemma \ref{K}, $H_i (K(\mathbb{Z}_p ,1))$  is nonzero for infinitely many values of $i$ which is a contradiction with $H_i (X)\cong \leqslant H_i (P)$ for all $i\geq 0$. So, $\tilde{X}$ has the homotopy type of $\mathbb{S}^3$. Also, the homomorphism  $(d_X )_* :H_i (P)\longrightarrow H_i (X)$ which is induced by the domination map $d_X$ is an epimorphism between two isomorphic Hopfian groups and so, is an isomorphism for all $i\geq 2$. Now by the Whitehead Theorem, we obtain that $d_X :P\longrightarrow X$ is a homotopy equivalence. Hence $X$ has the homotopy type of $P$.
\end{proof}

By a similar argument to the previous theorem, we can compute the capacity of a real projective $n$-space. Note that this result was proved by Y. Kodama et al. in \cite{Kod} in a different manner.
\begin{theorem}
The capacity of $\mathbb{RP}^n$ is equal to $2$, for all $n\geq 1$
\end{theorem}
\begin{proof}
Suppose that $X$ is homotopy dominated by $P=\mathbb{RP}^n$ with the domination map $d_X :P\longrightarrow X$ and $\tilde{X}$  denotes the universal covering space of $X$. Then $\pi_1 (X)$ and $H_{i}(X)$ are isomorphic to a direct suumand of $\pi_1 (P)\cong \mathbb{Z}_2$ and of $H_i (P)$, respectively, for all $i\geq 2$. We have the following cases:

Case One. $\pi_1 (X)\cong \mathbb{Z}_2$. Since $\tilde{X}$ is homotopy dominated by $\tilde{P}=\mathbb{S}^n$ and  $C(\mathbb{S}^n )=2$,  $\tilde{X}$ has the homotopy type of $\{ *\}$ or $\mathbb{S}^n$. If $\tilde{X}$ has the homotopy type of $\{ *\}$,  then  $X$ is the Eilenberg-MacLane space $K(\mathbb{Z}_2 ,1)$. By Lemma \ref{K}, $H_i (K(\mathbb{Z}_2 ,1))$  is nonzero for infinitely many values of $i$. But this is a contradiction with $H_i (X)\cong \leqslant H_i (P)$ for all $i\geq 0$. Hence $\tilde{X}$ has the homotopy type of $\mathbb{S}^n$ and so, $H_i (\tilde{X})\cong H_i (\mathbb{S}^n )$ for $i\geq 0$. Also, the homomorphism $(\tilde{d}_X)_* :H_i (\mathbb{S}^n )\longrightarrow H_i (\tilde{X})$ which is induced by the domination map $\tilde{d}_X$ is an epimorphism between two isomorphic Hopfian groups and so, it is an isomorphism for all $i\geq 2$. Thus  by the Whitehead Theorem, $d_X :P \longrightarrow X$ is a homotopy equivalence  and so, $X$ and $P$ have the same homotopy type.

Case Two. $\pi_1 (X)=1$.  First, let $n$ be even. We know that
\[
H_i (P)=
\begin{cases}
\mathbb{Z}, & i=0\\
\mathbb{Z}_2 , & i \; \text{is odd and} \; 0<i<n\\
0, & \text{otherwise}.
\end{cases}
\]
 Since $X$ is 1-connected, by Theorem \ref{1} $\pi_2 (X)\cong H_2 (X)$. But  $H_2 (X)=0$ and so, $X$ is 2-connected. Now again by Theorem \ref{1}, $\pi_3 (X)\cong H_3 (X)$. If $H_3 (X)\neq 0$, then $H_3 (X)\cong \mathbb{Z}_2$. Hence $\pi_3 (X)\cong \mathbb{Z}_2$. But this is a contradiction since $\pi_3 (X)\cong \leqslant \pi_3 (P)\cong \pi_3 (\mathbb{S}^n)=0$. Therefore, $H_3 (X)=0$ and so, $\pi_3 (X)=0$. Then $X$ is 3-connected and by Theorem \ref{1}, $\pi_4 (X) \cong H_4 (X)$. Therefore, $X$ is 4-connected because   $H_4 (X)\cong \leqslant H_4 (P)=0$. Again, by Theorem \ref{1} $\pi_5 (X)\cong H_5 (X)$. By continuing this process, we have $H_{i}(X)=0$ for all $1\leq i\leq n-1$. On the other hand, since $H_i (X)\cong \leqslant H_i (P)$, we have $H_i (X)=0$ for all $i\geq n$. Thus $X$ is a simply connected CW-complex with trivial homology groups and hence by the Whitehead Theorem, $X$ has the homotopy type of $\{ *\}$.

Second, let $n$ be odd. We know that
\[
H_i (P)=
\begin{cases}
\mathbb{Z}, & i=0,n\\
\mathbb{Z}_2 , & i \; \text{is odd and} \; 0<i<n\\
0, & \text{otherwise}.
\end{cases}
\]
 Similar to the previous argument, we obtain $H_i (X)=0$ for all $i\neq n$.  If $H_n (X)\neq 0$, then $H_n (X)\cong \mathbb{Z}$ and so $X$ has the homotopy type of $\mathbb{S}^n$. But $\mathbb{S}^n$ can not be homotopy dominated by $P$. Because if $\mathbb{S}^n$ is homotopy dominated by $P$, since $H^n (\mathbb{S}^n ;\mathbb{Z}_2)\neq 0$, then by Theorem \ref{4}, $\mathbb{S}^n$ has the homotopy type of $P$ which is a contradiction. Therefore $H_n (X)=0$ and so by the Whitehead Theorem, $X$ has the homotopy type of $\{ *\}$.
\end{proof}
\section{An Upper Bound for the Capacity of a $\mathbb{Z}_n$-Complex}
In this section, we present an upper bound for the capacity of a $\mathbb{Z}_n$-complex, i.e., a 2-dimensional CW-complex with finite cyclic fundamental group $\mathbb{Z}_n$.
\begin{definition}\cite{trees}.
A $\pi$-complex is a  2-dimensional CW-complex with the fundamental group $\pi$. The set $HT(\pi )$ denotes the set of all $\pi$-complexes.
\end{definition}
One can consider $HT(\pi)$  as a graph whose edges connect the type of each $\pi$-complex $X$ to the type of its sum $X \vee \mathbb{S}^2$ with the 2-sphere $\mathbb{S}^2$. These graphs are actually trees; they clearly contain no circuits, and they are connected because any two $\pi$-complexes have the same type once each is
summed with an appropriate number of copies of the 2-sphere $\mathbb{S}^2$ (see \cite{trees}).
\begin{definition}\cite{trees}.
A root is the homotopy type of a 2-dimensional CW-complex that does not admit a factorization involving an $\mathbb{S}^2$ summand.
\end{definition}
\begin{definition}\cite{trees}.
A junction is the homotopy type of a 2-dimensional CW-complex that admits two
or more inequivalent factorizations involving an $\mathbb{S}^2$ summand
\end{definition}
Indeed, the roots generate the rest of the types in the tree $HT(\pi )$ under the operation of
forming sum with $\mathbb{S}^2$ and the junctions determine the shape of the tree (see \cite{trees}).
\begin{lemma}\cite[p. 78]{Bau}.\label{com0}
The tree $HT (\mathbb{Z}_n )$ has exactly one root given by the pseudo projective plane $\mathbb{P}_n =\mathbb{S}^1 \cup_{f} e^2$ which is obtained by attaching a 2-cell $e^2$ to $\mathbb{S}^1$ via the map $f:\mathbb{S}^1 \longrightarrow \mathbb{S}^1$ of degree $n$.
\end{lemma}
\begin{lemma}\cite[Complement, p. 64]{Wall}.\label{com}
If $X$ satisfies $\mathcal{D}_2$, it is equivalent to a 3-dimensional CW-complex.
\end{lemma}
\begin{lemma}\cite[Corollary 2, p. 412]{Cohen}.\label{com1}
If $X$ is an $(n+1)$-complex dominated by an $n$-complex, then there exists  a wedge
of $k$ copies of $n$-spheres $W$, where $k$ is equal to the number of $(n+1)$-cells in $X$, such that $X \vee W \simeq X^{(n)}$, the $n$-skeleton of $X$.
\end{lemma}
\begin{lemma}\cite{trees}.\label{com2}
Let $X$ be a $\mathbb{Z}_n$-complex. Then $X$ has the homotopy type of the wedge sum $\mathbb{P}_n \vee \mathbb{S}^2 \vee \cdots \vee \mathbb{S}^2$ of the pseudo projective plane $\mathbb{P}_q$ and  rank $H_2 (X)$ copies of 2-sphere $\mathbb{S}^2$.
\end{lemma}
\begin{lemma}\cite{trees}.\label{com3}
The following are equivalent statements for a finitely presented group $\pi$.
\begin{enumerate}
\item
The tree $HT(\pi )$ of homotopy types of $\pi$-complexes has a single root.
\item
For $\pi$-complexes, there is a cancellation law for $\mathbb{S}^2$-summands, i.e.,
\[
X \vee \mathbb{S}^2 \simeq Y \vee \mathbb{S}^2 \quad \text{implies} \quad X\simeq Y.
\]
\end{enumerate}
\end{lemma}
We recall the following old problem concerning CW complexes.
\begin{center}
\textit{If $X$ is a CW-complex homotopy dominated by an $n$-dimensional CW-complex, then is $X$ homotopy equivalent to an $n$-dimensional CW-complex?}
\end{center}
C.T.C. Wall in \cite{Wall} showed that the answer is yes if $n>2$. The Stallings-Swan Theorem \cite{Swan} answers the problem affirmatively for $n = 1$. But the answer for the case $n=2$ is still unknown. J.M. Cohen in \cite{Cohen} showd that if $X$ is dominated by a 2-dimensional complex, then there is a wedge of 2-spheres $W$ such that $X\vee W$ is of the homotopy type of a 2-dimensional complex (for more detalis, see \cite{Cohen}). In the next theorem, we give a positive answer to the above question for $\mathbb{Z}_n$-complexes. In addition, we determine a space which is homotopy dominated by a $\mathbb{Z}_n$-complex up to homotopy equivalent and  using this result, we present an upper bound for the capacity of a $\mathbb{Z}_n$-complex.
\begin{theorem}\label{Main}
Let $n=p_{1}^{\alpha_1}p_{2}^{\alpha_2}\cdots p_{m}^{\alpha_m}$ where $p_i$'s are mutually distinct primes  and   $\alpha_i$'s are positive integers. Then every space homotopy dominated by a $\mathbb{Z}_n$-complex has the homotopy type of a $\mathbb{Z}_{m}$-complex where $m=p_{i_1}^{\alpha_{i_1}}\cdots p_{i_j}^{\alpha_{i_j}}$ for $i_1 ,\cdots ,i_j \in \{ 1,\cdots ,m\}$.
\end{theorem}
\begin{proof}
Suppose $P$ is a $\mathbb{Z}_n$-complex and  $X$ is homotopy dominated by $P$. By Lemma \ref{com}, we can suppose that $X$ is a 3-dimensional complex. It is easy to see that  $\pi_1 (X)$ is isomorphic to a direct summand of $\pi_1 (P)$. So we can suppose that $\pi_1 (X)=\mathbb{Z}_{m}$ where $m=p_{i_1}^{\alpha_{i_1}}\cdots p_{i_j}^{\alpha_{i_j}}$ for $i_1 ,\cdots ,i_j \in \{ 1,\cdots ,m\}$. By Lemma \ref{com1}, there exists a wedge of $k$ copies of 2-spheres $W$  such that $X \vee W \simeq X^{(2)}$. Since $\pi_1 (X^{(2)})= \pi_1 (X)= \mathbb{Z}_m$, so by Lemma \ref{com2} we have
\[
X\vee W \simeq X^{(2)}\simeq \mathbb{P}_m \vee \mathbb{S}^2 \vee \cdots \vee \mathbb{S}^2
\]
of rank $H_2 (X^{(2)})$ copies of the 2-sphere $\mathbb{S}^2$.

By the hypothesis, $X$ is homotopy dominated by $P$ which is a finite 2-dimensional polyhedron. Therefore, $H_2 (X)$ is isomorphic to a subgroup of the free abelian group $H_2 (P)$. Then $H_2 (X)$ is also a free abelian group of finite rank. On the other hand,  since $X \vee W \simeq X^{(2)}$, $H_2 (X)\oplus H_2 (W)\cong H_2 (X^{(2)})$.   Hence, we obtain
\[
rank\; H_2 (X)=rank\; H_2 (X^{(2)})-k.
\]
 Now by Lemmas \ref{com0} and \ref{com3}, $X\simeq \mathbb{P}_m \vee \mathbb{S}^2 \vee \cdots \vee \mathbb{S}^2$ of rank $H_2 (X)$ copies of the 2-sphere $\mathbb{S}^2$.
\end{proof}
\begin{corollary}
Let  $P$ be a $\mathbb{Z}_n$-complex where $n=p_{1}^{\alpha_1}p_{2}^{\alpha_2}\cdots p_{m}^{\alpha_m}$ for mutually distinct primes $p_i$  and positive integers $\alpha_i$.  Then $C(P)\leq 2^m \times (rank\; H_2 (P) +1)$.
\end{corollary}
\begin{proof}
By Therem \ref{Main}, every space homotopy dominated by  $P$ has the homotopy type of   $\mathbb{P}_m \vee \mathbb{S}^2 \vee \cdots \vee \mathbb{S}^2$, where $m=p_{i_1}^{\alpha_{i_1}}\cdots p_{i_j}^{\alpha_{i_j}}$ for $i_1 ,\cdots ,i_j \in \{ 1,\cdots ,m\}$,  with $k$ copies of the 2-sphere $\mathbb{S}^2$  for every $0\leq k\leq rank\; H_2 (P)$. Thus the proof is complete.
\end{proof}

\end{document}